\documentclass[submission]{eptcs}
\providecommand{\event}{SYCO 9} 

\usepackage{graphics}
\usepackage{graphicx} 
\usepackage{amsmath, amsthm}
\usepackage{mathtools}
\usepackage{stmaryrd} 
\usepackage{todonotes} 
\usepackage{quiver}
\usepackage{tikzit}

\tikzset{shorten <>/.style={shorten >=#1,shorten <=#1}}
\usetikzlibrary{decorations.markings}

\tikzstyle{circ}=[fill=white, draw=black, shape=circle]
\tikzstyle{blank}=[fill=white, draw=white, shape=circle]
\tikzstyle{none}=[fill=none, draw=none]
\tikzstyle{copy}=[fill=white, draw=black, shape=circle, minimum height=0.2cm, inner sep=0]
\tikzstyle{varCopy}=[fill=black, draw=black, shape=circle, minimum height=0.2cm, inner sep=0]
\tikzstyle{delete}=[fill=black, draw=black, shape=circle, minimum height=0.2cm, inner sep=0]
\tikzstyle{1morph1}=[fill=white, draw=black, shape=rectangle, minimum width=1cm, minimum height=1cm]
\tikzstyle{1morph}=[fill=white, draw=black, shape=rectangle, minimum width=0.75cm, minimum height=0.75cm, inner sep=0.1cm]
\tikzstyle{2morph2}=[fill=white, draw=black, shape=rectangle, minimum width=1cm, minimum height=2cm]
\tikzstyle{2morph}=[fill=white, draw=black, shape=rectangle, minimum width=1cm, minimum height=1.25cm, inner sep=0.1cm]
\tikzstyle{nmorph}=[fill=white, draw=black, shape=rectangle, minimum height=6cm, minimum width=1cm, inner sep=0.1cm]
\tikzstyle{1state}=[fill=white, draw=black, regular polygon, regular polygon sides=3, minimum height=0.5cm, regular polygon rotate=-30]
\tikzstyle{dbox}=[fill=white, draw=black, dashed, shape=rectangle, minimum width=2cm, minimum height=1cm, inner sep=0.1cm]
\tikzstyle{vdbox}=[fill=white, draw=black, dashed, shape=rectangle, minimum width=2cm, minimum height=1.5cm, inner sep=0.1cm]
\tikzstyle{bigbox}=[fill=white, draw=black, dashed, shape=rectangle, minimum width=2cm, minimum height=4cm, inner sep=0.1cm]
\tikzstyle{2state}=[inner sep=0.05cm, fill=white, draw=black, isosceles triangle, minimum width=1.25cm, isosceles triangle apex angle=90, shape border rotate=180]
\tikzstyle{var2state}=[inner sep=0.05cm, fill=white, draw=black, isosceles triangle, minimum width=1.25cm, isosceles triangle apex angle=60, shape border rotate=180]
\tikzstyle{g2state}=[inner sep=0.05cm, fill=white, draw=black, isosceles triangle, minimum width=6cm, isosceles triangle apex angle=110, shape border rotate=180]
\tikzstyle{bigstate}=[inner sep=0.05cm, fill=white, draw=black, isosceles triangle, minimum width=3cm, isosceles triangle apex angle=110, shape border rotate=180]
\tikzstyle{bigeffect}=[inner sep=0.05cm, fill=white, draw=black, isosceles triangle, minimum width=3cm, isosceles triangle apex angle=110]
\tikzstyle{g2effect}=[inner sep=0.05cm, fill=white, draw=black, isosceles triangle, minimum width=6cm, isosceles triangle apex angle=110]
\tikzstyle{2effect}=[inner sep=0.05cm, fill=white, draw=black, isosceles triangle, minimum width=1.25cm, isosceles triangle apex angle=90]
\tikzstyle{b2effect}=[inner sep=0.05cm, fill=white, draw=black, isosceles triangle, minimum width=2cm, isosceles triangle apex angle=90]
\tikzstyle{Medium box}=[fill=white, draw=black, shape=rectangle, tikzit shape=rectangle, minimum width=1.5cm, minimum height=1.5cm]
\tikzstyle{equals}=[double, thick, -]

\tikzstyle{midArrow}=[-, decoration={{markings,mark=at position .5 with {\arrow{>}}}}, postaction=decorate]
\tikzstyle{Black arrow}=[->]
\tikzstyle{Gray line}=[-, draw={rgb,255: red,191; green,191; blue,191}, line width=0.8]
\tikzstyle{arrow}=[->]

\theoremstyle{definition}
\newtheorem{definition}{Definition}

\newtheorem{remark}[definition]{Remark}

\theoremstyle{plain}

\newtheorem{proposition}[definition]{Proposition}
\newtheorem{lemma}[definition]{Lemma}
\newtheorem{corollary}[definition]{Corollary}

\newcommand{\asequal}[1]{\ensuremath{\cong_{#1}}}
\newcommand{\Ca}{\mathcal{C}}
\newcommand{\Set}{\mathbf{Set}}
\newcommand{\Cat}{\mathbf{Cat}}
\newcommand{\Fam}{\mathbf{Fam}}
\newcommand{\comp}{\fatsemi}
\newcommand{\id}{\mathrm{id}}
\newcommand{\inv}{\ensuremath{^{-1}}}
\newcommand{\op}{\ensuremath{^{\mathrm{op}}}}
\newcommand{\dup}{\operatorname{copy}}

\newcommand{\BayesChart}{\mathbf{BayesChart}}
\newcommand{\BayesLens}{\mathbf{BayesLens}}
\newcommand{\diset}[2]{\binom{#1}{#2}}
\newcommand{\chartto}{\rightrightarrows}
\newcommand{\lensto}{\rightleftarrows}
\newcommand{\Stat}{\mathbf{Stat}}

\title{Dependent Bayesian Lenses: Categories of Bidirectional Markov Kernels with Canonical Bayesian Inversion}
\author{Dylan Braithwaite
\institute{MSP Group \\ University of Strathclyde \\ Glasgow, Scotland}
\email{dylan.braithwaite@strath.ac.uk}
\and
Jules Hedges
\institute{MSP Group \\ University of Strathclyde \\ Glasgow, Scotland}
\email{jules.hedges@strath.ac.uk}
}

\def\titlerunning{Dependent Bayesian Lenses}
\def\authorrunning{D. Braithwaite, J. Hedges}
\begin{document}
\maketitle

\begin{abstract}
We generalise an existing construction of Bayesian Lenses to admit lenses between pairs of objects where the backwards object is dependent on states on the forwards object (interpreted as probability distributions). This gives a natural setting for studying stochastic maps with Bayesian inverses restricted to the points supported by a given prior. In order to state this formally we develop a proposed definition by Fritz
\cite{synthetic_approach}
of a support object in a Markov category and show that these give rise to a section into the category of dependent Bayesian lenses encoding a more canonical notion of Bayesian inversion.
\end{abstract}

\section{Introduction}

Categories of lenses provide models of bidirectional transformations between objects in a cartesian category.
However this fails to generalise to non-cartesian monoidal categories, which are of interest when studying stochastic processes.
The most common generalisation of lenses to these settings is via categories of optics, and indeed optics have been used for probabilistic settings, for example in Bayesian open games \cite{bayesian_open_games}.

Bayesian open games however, are defined relative to a fixed prior, but when studying systems with Bayesian updating in general we often want to quantify over this prior, resulting in open systems with families of inverses, parameterised by the choice of prior.
The approach taken in open games, of using optics over a Markov category, is not compatible with this form of parameterisation because Bayesian inversion with the prior as a parameter cannot in general be defined internal to the category.
In light of this, an alternative way to model Bayesian processes in a lens-like fashion is described in \cite{statistical_games}.
This uses an alternative generalisation of lenses, first described by \cite{spivak_lenses}, which notes that many categories resembling lenses can be constructed as Grothendieck constructions of pointwise opposites of indexed categories.

Bayesian lenses include lenses whose forward component is a stochastic map and whose backward component is a (family of) Bayesian inverses to the forwards component.
Indeed Smithe shows in \cite{bayesian_updates} that a Markov category embeds functorially into its category of Bayesian lenses, but this has a shortcoming in the fact that Bayesian inverses are not in general unique.
Specifically, the abstract definition for Bayesian inversion in a Markov category, due to \cite{bayesian_inversion}, does not uniquely specify a morphism because it allows for the behaviour of the map to be arbitrary on points not supported by the prior.
Hence, any such embedding functor necessitates a coherent choice of inverses for each morphism-prior pair.

In this paper we propose an a modified definition for a Bayesian inverse in a Markov category using a notion of \emph{support object}, based on a definition proposed in \cite{synthetic_approach}.
In this case Bayesian inverses between support objects are indeed unique and so give rise to a canonical Bayesian inversion functor.
To acomodate this new definition we propose a definition for \emph{dependent Bayesian lenses} where the backward object is allowed to depend on a choice of distribution over the forward object.

Having decided on the structure required from dependent Bayesian lenses, it is possible to directly modify the non-dependent version of Bayesian lenses to obtain our definition, but here we take another perspective to motivate the definition, by first considering families of support objects.
In many cases where we want to work with an object supported at an arbitrary prior, it is useful to instead consider a family of objects indexed by the collection of all possible priors.
Formalising this using the family fibration \cite[Example 8.1.9b]{fam_fibration}, we obtain an indexed category which closely resembles the $\mathbf{Stat}$ construction used in defining standard Bayesian lenses.
This not only gives a neat way of defining an indexed category for Bayesian lenses, but also justifies calling these dependent lenses, by analogy with the uses of the family fibration in dependent type theory \cite{jacobs_categorical_logic}.

\section{Markov Categories with Supports}

We begin by recalling the definition of a Markov category, due to Fritz in \cite{synthetic_approach}.
\begin{definition}
A Markov category is a symmetric monoidal category with a supply of commutative comonoids satisfying compatibility equations:
\end{definition}

\ctikzfig{./tikz/markov_cat.tikz}

This provides a simple axiomatisation for categories of probability spaces, including as examples the Kleisli category of various probability monads (such as the monad sending sets to the set of their finitely supported probability distributions), and categories of matrices with Gaussian noise.
Other examples include various measure-theoretic settings for probability, but the above examples, $\mathbf{FinStoch}$ and $\mathbf{Gauss}$, are notable in this paper because they can be easily seen to admit all support objects for states $I \to X$, as we shall discuss later.

Many important concepts from probability theory can be stated in an abstract form in a Markov category, including almost-equality and Bayes law.

\begin{definition}[Almost-sure equality]
Fix morphisms $\pi: I \to X$ and $f, g: X \to Y$.
We say that $f$ is $\pi$-almost equal to $g$ (written $f \asequal\pi g$) if there is the following equation of morphisms:
\ctikzfig{./tikz/as_equal.tikz}
\end{definition}

\begin{definition}[Bayesian Inversion]
Fix morphisms $\pi: I \to X$ and $f: X \to Y$.
A Bayesian inverse for $f$ at $\pi$ is a morphism $f^\dag_\pi: Y \to X$ satisfying the following equation:
\ctikzfig{./tikz/bayesian_inverse.tikz}
\end{definition}

Giving an abstract account of Bayes' law, the latter definition should be very important in studying Bayesian statistics categorically, but it is in some way unsatisfying because it only specifies a morphism up to almost-sure equality.
For example, considering $\mathbf{FinStoch}$, if a distribution $\pi: I \to X$ is not fully supported, then the image of a morphism $X \to Y$ is only specified by the above definition at the points in the support of $\pi$.
We can work around this ambiguity however by instead considering inverses as morphisms between objects representing the supports of distributions.
\cite{synthetic_approach} proposes a definition for support objects in a Markov category, but does not develop the idea further.
In this section we investigate some properties of support objects and give some examples of Markov categories with supports for every distribution.

\begin{definition}
Fix a state $\pi: I \to X$. An object $X_\pi$ is called a \emph{support of $\pi$} if $X_\pi$ represents the covariant functor $(\Ca(X, -) / \asequal \pi): \Ca \to \Set$.
\end{definition}

This definition seems to succinctly capture the essential properties of the support of a distribution, but it is quite opaque and does not encourage intuition.
We can however restate this as an equivalent condition involving the existence of restriction and inclusion morphisms for the object.

\begin{proposition}
$X_\pi$ is a support of $\pi: I \to X$ if and only if there is a section-retraction pair $X_\pi \xrightarrow{i} X \xrightarrow{r} X_\pi $ such that for any morphisms $f, g: X \to Y$ we have $f \asequal\pi g \iff i\comp f = i \comp g$.
\end{proposition}

\begin{proof}
Assume $X_\pi$ is a support. This means we have a natural isomorphism 
$\Phi: (\Ca(X, -)/\asequal\pi) \to \Ca(X_\pi, -) $.
We take $i = \Phi_X(\id_X)$, then we see from the following naturality square that the action of $\Phi$ must be to precompose representative morphisms with $i$:
\[\begin{tikzcd}
	\textcolor{rgb,255:red,128;green,128;blue,128}{i} &&& \textcolor{rgb,255:red,130;green,130;blue,130}{[\id_X]_{\asequal\pi}} \\
	& {\Ca(X_\pi, X)} & {\Ca(X, X)/\asequal\pi} \\
	& {\Ca(X_\pi, Y)} & {\Ca(X, Y)/\asequal\pi} \\
	\textcolor{rgb,255:red,128;green,128;blue,128}{i\comp f} &&& \textcolor{rgb,255:red,128;green,128;blue,128}{[f]_{\asequal\pi}}
	\arrow["{\Ca(X, f)/\asequal\pi}", from=2-3, to=3-3]
	\arrow["{\Ca(X_\pi, f)}"', from=2-2, to=3-2]
	\arrow["{\Phi_X}"', from=2-3, to=2-2]
	\arrow["{\Phi_Y}", from=3-3, to=3-2]
	\arrow[draw={rgb,255:red,128;green,128;blue,128}, shorten <=11pt, shorten >=11pt, maps to, from=1-4, to=4-4]
	\arrow[draw={rgb,255:red,128;green,128;blue,128}, shorten <=33pt, shorten >=33pt, maps to, from=4-4, to=4-1]
	\arrow[draw={rgb,255:red,128;green,128;blue,128}, shorten <=11pt, shorten >=11pt, maps to, from=1-1, to=4-1]
	\arrow[draw={rgb,255:red,128;green,128;blue,128}, shorten <=17pt, shorten >=33pt, maps to, from=1-4, to=1-1]
\end{tikzcd}\]
Hence we establish the property that precomposition by $i$ is an isomorphism between morphisms from $X$ and $\asequal\pi$-equivalences classes of morphisms from $X_\pi$.
We further have that $\id_{X_\pi} = \Phi(\Phi\inv(\id_{X_\pi})) = i\comp\Phi\inv(\id_{X_\pi})$, so we can take the retract to be $r = \Phi\inv(\id)$.

Conversely, given such an $i$ and $r$, it is clear that precomposition by $i$ defines a function $\Ca(X, Y) \to \Ca(X_\pi, Y)$ natural in $Y$ and the assumed property of $i$ guarantees that this is a bijection from $\asequal\pi$-equivalence classes.
Finally we have that $i\comp r\comp i = i$, so $r\comp i \asequal\pi \id_{X_\pi}$.
Hence precomposition by $r$ is an inverse to $i\comp (-): (\Ca(X, -)/\asequal\pi) \to \Ca(X_\pi, -)$.
\end{proof}
This is discussed further in \cite{structural_foundations} as well as in upcoming work \cite{upcoming_markov_cat_paper}.

While support objects are not necessarily unique, we must have that they are unique up to isomorphism, since two support objects for the same distribution must by definition represent the same presheaf.
When we disuss a fixed support object we really mean a fixed support object with a choice of section and retraction (or equivalently a choice of the representing isomorphism $\Phi$).

Now, if we have a distribution on $X$, $\pi: I \to X$, and a morphism $f: X \to Y$ we can push this forward to a distribution $\pi\comp f$ on $Y$.
So $f$ restricts to a morphism $f_\pi: X_{\pi} \to Y_{\pi\comp f}$ defined by $i_\pi \comp f \comp r_{\pi\comp f}$,
where $i_\pi$ is the section for $X_\pi$ and $r_{\pi\comp f}$ is the retraction for $Y_{\pi\comp f}$.
Based on this there is an obvious adjustment to the definition of Bayesian inversion in order to capture Bayesian inverses between supports:

\begin{definition}[Bayesian-inverse-with-support]
Fix morphisms $\pi: I \to X$ and $f: X \to Y$ with support objects $X_\pi$ and $Y_{\pi\comp f}$.
We call a morphism $f^\sharp_\pi: Y_{\pi\comp f} \to X_\pi$  a \emph{Bayesian inverse with support} if we have the following equation of morphisms:
\ctikzfig{./tikz/bayesian_inverse_with_support.tikz}
\end{definition}

And we can show that this indeed captures the same concept as the earlier version:

\begin{proposition}
\label{bayesian_inverse_bijection}
Fix morphisms $f: X \to Y$ and $\pi: I \to X$, and support objects $X_\pi$ and $Y_{f\comp\pi}$.
Then inverses-with-support of $f$ at $\pi$ are in bijection with $\asequal{\pi\comp f}$-equivalence classes of ordinary Bayesian inverses.
\end{proposition}
\begin{proof}

We first exhibit a map $\Psi$ from inverses-with-support to ordinary inverses.
Let $g: Y_{\pi\comp f} \to X_\pi$ be a Bayesian inverse with support of $f$.
By definition this means that $\Psi(g) = r_{\pi\comp f}\comp  g\comp i_\pi$ is an ordinary Bayesian inverse.

Conversely if $h: Y \to X$ is an ordinary Bayesian inverse, then we define $\tilde{\Psi}(h)$ to be the composition
\[
S_{\pi\comp f} 
\xrightarrow{i_{\pi\comp f}}
Y
\xrightarrow{h}
X
\xrightarrow{r_\pi}
S_\pi.
\]
We can see that this is an inverse-with-support:
\ctikzfig{./tikz/bayesian_inverse_equiv_proof.tikz}
We finally have that $\Psi(-)$ is inverse to $\tilde\Psi$ when viewed as maps to/from equivalence classes:
\begin{alignat*}{3}
	\Psi(\tilde\Psi(h)) &=& r_{\pi\comp f}\comp \tilde\Psi(h)\comp i_\pi\\
	  &=& r_{\pi\comp f}\comp i_{\pi\comp f} \comp h \comp r_\pi \comp i_\pi \\
	  &\asequal{\pi\comp f}& h \comp r_\pi \comp i_\pi \\
	  &\asequal{\pi\comp f}& h
\end{alignat*}	
where the final equivalence uses the fact that $h$ is a Bayesian inverse to move it out of the way, similarly to the previous chain of equalities.
\end{proof}

Noting that all Bayesian inverses to $f$ at $\pi$ must be $(\pi\comp f)$-almost equal we have as a corollary that Bayesian inverses with support are unique.

\begin{corollary}
Fix morphisms $f: X \to Y$ and $\pi: I \to X$, and support objects $X_\pi$ and $Y_{f\comp\pi}$.
Then there is at most one inverse-with-support to $f$ at $\pi$.
\qed
\end{corollary}

This final result suggests that we may be able to define a (contravariant) functor that picks out the canonical inverse for a given morphism.
However, any such functor would necessitate a coherent assignment of priors to objects.
Really we would like a setting where we can work with inversion at arbitrary priors.
In order to define a well defined functor of this sort we can instead work with families of support objects, indexed by the distribution they are supporting.
We formalise this in the next section as an indexed category $\Ca \to \Cat$ sending $X$ to $\Ca(I, X)$-indexed families of objects.

\section{Families of Support Objects}

If $\Ca$ is a Markov category with all support objects then we can consider \emph{families of supports} at an object.
We define this in terms of the indexed category corresponding to the families fibration \cite{fam_fibration}. 
Namely this sends a set to a category of objects indexed over that set.
Specifically $\Fam_\Ca: \Set\op \to \Cat$ sends $X$ to the category whose objects are $X$-indexed families of objects of $\Ca$, and whose morphisms $A \to B$ are given by a function $f: X \to X$ and a family of $\Ca$-morphisms $A(x) \to B(f(x))$.
Alternatively this is the category $\Ca^X$ of functors $X \to \Ca$ where $X$ is viewed as a discrete category.
Then for a function $g: X \to Y$, $\Fam_\Ca(g)$ is a functor $\Ca^Y \to \Ca^X$ which acts by precomposition.

Using this, we have an indexed category $S: \Ca\op \to \Cat$ defined as so:
\[\begin{tikzcd}
	\Ca\op & \Set\op & \Cat \\
	X & {\Ca(I, X)} & {\Ca^{\Ca(I, X)}} \\
	Y & {\Ca(I, Y)} & {\Ca^{\Ca(I, Y)}}
	\arrow["{\Ca(I, -)\op}", from=1-1, to=1-2]
	\arrow["{\Fam_{\Ca}}", from=1-2, to=1-3]
	\arrow["f"', from=2-1, to=3-1]
	\arrow["{\Ca(I, f)}"', from=2-2, to=3-2]
	\arrow["{\Ca^{\Ca(I, f)}}"', from=3-3, to=2-3]
	\arrow[maps to, from=2-1, to=2-2]
	\arrow[maps to, from=3-1, to=3-2]
	\arrow[maps to, from=2-2, to=2-3]
	\arrow[maps to, from=3-2, to=3-3]
\end{tikzcd}\]
For an object $X \in \Ca$ we have that the objects of $S(X)$ are $\Ca(I, X)$-indexed families of objects.
But an inverse-with-support for $f: X \to Y$ must move between the fibres $S(X)$ and $S(Y)$.
Hence in order to represent Bayesian inverses we take the Grothendieck construction $\int \Fam_\Ca(\Ca(I, -)\op)$.
Following the conventions of Myers' program of categorical systems theory
\cite{djm_dynamical_systems}
\cite{categorical_systems_theory}
, and anticipating the later section where we will consider the fibrewise opposite of this category as a category of lenses, we refer to this as the \emph{category of (dependent) Bayesian charts}:
\[\BayesChart(\Ca) = \int_{X \in \Ca} \Fam_\Ca(\Ca(I, X)\op).\]
This category has as objects, pairs $\diset X A$ where $X$ is an object of $\Ca$ and $A$ is a $\Ca(I, X)$-indexed family of objects of $\Ca$.
Morphisms $\diset X A \chartto \diset Y B$ are given by a morphism $f: X \to Y$ in $\Ca$ and a morphism $f^\flat_\pi: A(\pi) \to B(\pi\comp f)$ for each $\pi \in \Ca(I, X)$.

Viewing Bayesian charts as a fibred-category, we can embed $\Ca$ into this category via a section of the bundle.

\begin{proposition}
The bundle of $\BayesChart(\Ca)$ over $\Ca$ has a section $T: \Ca \to \BayesChart(\Ca)$ which maps objects to families of supports and morphisms into their restrictions to the respective supports:
\end{proposition}

\begin{proof}
We choose, for each morphism $f: I \to X$, a fixed support object $X_\pi$ with a section and retraction $i_\pi$ and $r_\pi$.
For an object $X$ we have $T(X) = \diset X {X_{(-)}}$
where $X_{(-)}$ denotes the family sending $\pi \in \Ca(I, X)$ to the support object of $X$ at $\pi$.
For $f: X \to Y$, we have $T(f) = \diset f {f_{(-)}}$ where $f_{(-)}$ is the family of morphisms sending $\pi$ to the morphism $f_\pi: X_\pi \to Y_{\pi\comp f}$ defined by $i_\pi\comp f \comp r_{\pi\comp f}$.
Focusing only on the backward objects this is summarised below:
\[\begin{tikzcd}
	\Ca & {\BayesChart(\Ca)} \\
	X & {X_{(-)}} \\
	Y & {Y_{(-)\comp f}}
	\arrow[""{name=0, anchor=center, inner sep=0}, "{i_{(-)}\comp f \comp r_{(-)\comp f}}", from=2-2, to=3-2]
	\arrow[""{name=1, anchor=center, inner sep=0}, "f"', from=2-1, to=3-1]
	\arrow["T", from=1-1, to=1-2]
	\arrow[shorten <=11pt, shorten >=11pt, maps to, from=1, to=0]
\end{tikzcd}\]
The image of the identity morphism $\id_X$ is, for each $\pi$, $i_\pi\comp \id_X\comp r_\pi$ which is equal to $\id_{X_\pi}$.
So we have $T(\id_X) = \diset {\id_X} {\id_{X_{(-)}}}$.

For functoriality we must show that, for $f: X \to Y$, $g: Y \to Z$, and $\pi: I \to X$,
we have $f_\pi\comp g_{\pi\comp f} = (f\comp g)_{\pi}$.
Specifically this means that
\[i_\pi \comp f\comp r_{\pi\comp f} \comp i_{\pi \comp f} \comp g \comp r_{\pi \comp f \comp g} = i_\pi \comp f \comp g \comp r_{\pi \comp f \comp g}.\]
To see this, note that this equation holds if and only if
$f\comp r_{\pi\comp f} \comp i_{\pi \comp f} \comp g \comp r_{\pi \comp f \comp g} \asequal\pi f \comp g \comp r_{\pi \comp f \comp g}.$
Moreover, from lemma \ref{asequal_base_change}, this is the case exactly when 
$r_{\pi\comp f} \comp i_{\pi \comp f} \comp g \comp r_{\pi \comp f \comp g} \asequal{\pi\comp f} g \comp r_{\pi \comp f \comp g}$
and so we are done.
\end{proof}

This pairing of an indexed category with a section gives what Myers refers to as a dynamical system doctrine.

When working in categorical probability an important feature is the ability to `copy' probability spaces, in the form of a comonoid structure on the category.
In order to extend that to the setting of Bayesian charts we will show that the functor $T$ described above is infact oplax-monoidal and so preserves comonoids.
But to do so we must first establish a monoidal product on Bayesian charts.
We define this monoidal product by showing how the indexed category defining $\BayesChart(\Ca)$ can be lifted to an indexed monoidal category.
Then Bayesian charts can be constructed as a monoidal category by the monoidal Grothendieck construction \cite{monoidal_grothendieck}.

An indexed monoidal category is equivalently a lax monoidal pseudofunctor $\Ca\op \to \Cat$.
Recall that a lax monoidal functor $F$ is one with a monoidal transformation $\mu_{X, Y}: F(X) \otimes F(Y) \to F(X \otimes Y)$ (the \emph{product laxator}), and a morphism $\epsilon: I \to F(I)$ (the \emph{unit laxator}), satisfying certain coherence conditions.
The dualised version of this is an oplax monoidal functor. Precisely, an oplax monoidal structure on $F: \mathcal X \to \mathcal Y$ is equivalently a lax monoidal structure on $F\op: \mathcal X\op \to \mathcal Y\op$.

To define laxators for our indexing functor $\Fam(\Ca(I, -)\op)$ we use the following lemma:

\begin{lemma}
The functor $\Ca(I, -): \Ca \to \Set$ can be made into an oplax monoidal functor with the product oplaxator given by marginalisation:
\[\begin{tikzcd}
	{\Ca(I, X \otimes Y)} & {\Ca(I, X) \times \Ca(I, Y)} \\
	\pi & {(\pi_L, \pi_R)}
	\arrow["{\delta_{X, Y}}", from=1-1, to=1-2]
	\arrow[maps to, from=2-1, to=2-2]
\end{tikzcd}\]
where $\pi_L$ and $\pi_R$ are the marginals of $\pi$ depicted below:
\ctikzfig{./tikz/marginalisation.tikz}
\end{lemma}
\begin{proof}
Since the monoidal unit in $\Set$	 is terminal, the unit oplaxator is trivial.

The naturality and coherence conditions can all be proved by diagram chases using naturality conditions of the structure maps, but can more easily be seen by noting that in each equation both sides trivially denote isotopic string diagrams, and so are equal by Joyal and Street's coherence theorem for symmetric monoidal categories \cite{coherence_theorems}.
As such we will not give proofs in any more detail, but for reference we do state explicitly the equations required.

Naturality for $\delta$ requires that for each $f: X \to X'$ and $g: Y \to Y'$, the following square commutes:
\[\begin{tikzcd}
	{\Ca(I, X \otimes Y)} & {\Ca(I, X) \times \Ca(I, Y)} \\
	{\Ca(I, X'\otimes Y')} & {\Ca(I, X') \times \Ca(I, Y')}
	\arrow["{\delta_{X, Y}}", from=1-1, to=1-2]
	\arrow["{\Ca(I, f \otimes g)}"', from=1-1, to=2-1]
	\arrow["{\delta_{X', Y'}}"', from=2-1, to=2-2]
	\arrow["{\Ca(I, f) \times \Ca(I, g)}", from=1-2, to=2-2]
\end{tikzcd}\]
i.e. that $(\pi\comp (f \otimes g))_L = \pi_L\comp f$ and $(\pi\comp (f \otimes g))_R = \pi_R\comp g$ for any $\pi: I \to X \otimes Y$.

The associativity condition requires commutativity of the following diagram
\[\begin{tikzcd}
	{(\Ca(I, X) \times \Ca(I, Y)) \times \Ca(I, Z)} & {\Ca(I, X) \times (\Ca(I, Y) \times \Ca(I, Z))} \\
	{\Ca(I, X \otimes Y) \otimes Z)} & {\Ca(I, X) \times \Ca(I, Y \otimes Z)} \\
	{\Ca(I, (X \otimes Y) \otimes Z)} & {\Ca(I, X \otimes (Y \otimes Z))}
	\arrow["\sim"', from=1-2, to=1-1]
	\arrow["{\Ca(I, X) \times \delta_{Y, Z}}"', from=2-2, to=1-2]
	\arrow["{\delta_{X, Y \otimes Z}}"', from=3-2, to=2-2]
	\arrow["{\Ca(I, \alpha_{X, Y, Z})}", from=3-2, to=3-1]
	\arrow["{\delta_{X \otimes Y, Z}}", from=3-1, to=2-1]
	\arrow["{\delta_{X, Y} \times \Ca(I, Z)}", from=2-1, to=1-1]
\end{tikzcd}\]
This means that 
\begin{itemize}
	\item $\pi_L = (\pi\comp\alpha_{X, Y, Z})_{LL}$,
	\item $\pi_{RL} = (\pi\comp\alpha_{X, Y, Z})_{LR}$,
	\item and $\pi_{RR} = (\pi\comp\alpha_{X, Y, Z})_R$,
\end{itemize}
for any $\pi: I \to X$.

The coherence for the left unitor requires that
\[\begin{tikzcd}
	{\{*\} \times \Ca(I, X)} & {\Ca(I, I) \otimes \Ca(I, X)} \\
	{\Ca(I, X)} & {\Ca(I, I \otimes X)}
	\arrow["{\Ca(I, \lambda_X)}"', from=2-1, to=2-2]
	\arrow["{\delta_{I, X}}"', from=2-2, to=1-2]
	\arrow["{\eta \otimes \Ca(I, X)}"', from=1-2, to=1-1]
	\arrow["\sim", from=2-1, to=1-1]
\end{tikzcd}\]
commutes.
Or equivalently that $(\pi\comp\lambda_X)_R = \pi$.
Similarly, the coherence for the right unitor requires that $(\pi\comp\rho_X)_L = \pi$.
\end{proof}

As such we have that $\Ca(I, -)\op$ can be made into a lax monoidal functor. Moreover, \cite{monoidal_grothendieck} shows that when $\Ca$ is monoidal, $\Fam_\Ca$ has a canonical lax monoidal structure given by pointwise tensoring, and so we have a lax monoidal functor $\Fam_\Ca(\Ca(I, -)\op)$.

Then by the monoidal Grothendieck construction we have the following corollary:

\begin{corollary}
$\BayesChart(\Ca)$ may be equipped with the structure of a monoidal category, with a tensor product given by
\[\diset{X}{A} \otimes \diset{Y}{B} = \diset{X \otimes Y}{A \otimes B}\]
where $A \otimes B: \Ca(I, X \otimes Y) \to \Ca$ sends $\pi: I \to X \otimes Y$ to $A(\pi_L)\otimes B(\pi_R)$.
\qed
\end{corollary}

With this product defined, we can state the final main result of this section:

\begin{proposition}
$T$ can be equipped with the structure of an oplax-monoidal functor.
\end{proposition}
\begin{proof}
Noting that there is only a single prior on $I$ given by the unique map $!: I \to I$, and that any support object $I_!$ must also be terminal, since there is a unique map $X \xrightarrow{!} I \xrightarrow{r_!} I_!$, we have that $T(I)$ is strictly a monoidal unit. So we can take the unit oplaxator to be the identity morphism $\id_{T(I)}$.

Now consider a pair of objects $X$ and $Y$.
We have $T(X) \otimes T(Y) = \diset {X \otimes Y} {X_{(-)} \otimes Y_{(-)}}$
where $X_{(-)} \otimes Y_{(-)}$ is a family of $\Ca(I, X \otimes Y)$-indexed objects sending $\pi$ to $X_{\pi_L}\otimes Y_{\pi_R}$.
We note that this is in general distinct from the support objects $(X \otimes Y)_{\pi}$ picked out by $T(X \otimes Y) = \diset{X \otimes Y} {(X \otimes Y)_{(-)}}$, but there is a morphism $\gamma_{X, Y}: (X \otimes Y)_\pi \to X_{\pi_L} \otimes Y_{\pi_R}$ defined by $i_\pi\comp(r_{\pi_L}\otimes r_{\pi_R})$ which we shall take to be the product-oplaxator of $T$.

We are required to show that this is a natural transformation between the functors
$(T\times T) \comp \otimes$ and $\otimes\comp T$.
This amounts to the following equality of string diagrams:
\ctikzfig{./tikz/oplaxator_natural.tikz}
which can be shown equal via straightforward calculation.

Showing coherence for the associators amounts to proving that the following diagram commutes in $\Ca$:
\[\begin{tikzcd}
	{(X_{\pi_{LL}} \otimes Y_{\pi_{LR}}) \otimes Z_{\pi_R}} && {X_{\pi_{LL}} \otimes (Y_{\pi_{LR}} \otimes Z_{\pi_R})} \\
	{(X \otimes Y)_{\pi_L} \otimes Z_{\pi_R}} && {X_{\pi_{LL}} \otimes (Y \otimes Z)_{(\pi\comp \alpha)_R}} \\
	{((X \otimes Y) \otimes Z)_\pi} && {(X \otimes (Y \otimes Z))_{\pi\comp \alpha}}
	\arrow["{\alpha_{X_{\pi_{LL}}, Y_{\pi_{LR}}, Z_{\pi_R}}}", from=1-1, to=1-3]
	\arrow["{X_{\pi_{LL}} \otimes (i_{(\pi\comp \alpha)_R}\comp(r_{\pi_{LR}}\otimes r_{\pi_R}))}"', from=2-3, to=1-3]
	\arrow["{i_{\pi\comp \alpha} \comp (r_{\pi_{LL}}\otimes r_{(\pi\comp\alpha)_R})}"', from=3-3, to=2-3]
	\arrow["{i_\pi\comp \alpha_{X, Y, Z} \comp r_{\pi\comp \alpha}}"', from=3-1, to=3-3]
	\arrow["{i_\pi\comp(r_{\pi_L} \otimes r_{\pi_R})}", from=3-1, to=2-1]
	\arrow["{(i_{\pi_L} \comp (r_{\pi_{LL}} \otimes r_{\pi_{LR}})) \otimes Z_{\pi_R}}", from=2-1, to=1-1]
\end{tikzcd}\]
which can be seen to commute via a similar string diagram calculation.

Finally the unitor coherence conditions amount to commutativity of the following diagrams for all $\pi: I \to I \otimes X$ and $\sigma: I \to X \otimes I$ and are true due to the naturality of the unitors:

\[\begin{tikzcd}
	& {I \otimes X_{\pi_R}} &&&& {X_{\sigma_L} \otimes I} \\
	{X_{\pi_R}} && {(I \otimes X)_{\pi}} && {X_{\sigma_L}} && {(X \otimes I)_\pi}
	\arrow["{\gamma_{I, X} = i_\pi\comp (I \otimes r_{\pi_R})}"', from=2-3, to=1-2]
	\arrow["{\lambda_{X_{\pi_R}}}"', from=1-2, to=2-1]
	\arrow["{i_\pi\comp \lambda_X \comp r_{\pi_R}}", from=2-3, to=2-1]
	\arrow["{\gamma_{X, I} = i_{\pi}\comp (r_{\sigma_L} \otimes I)}"', from=2-7, to=1-6]
	\arrow["{\rho_{X_{\sigma_L}}}"', from=1-6, to=2-5]
	\arrow["{i_\sigma \comp \rho_X \comp r_{\sigma_L}}", from=2-7, to=2-5]
\end{tikzcd}\]

\end{proof}

\begin{corollary}
Every object in the image of $T$ is a comonoid.
\qed
\end{corollary}

\section{Dependent Bayesian Lenses}

As alluded to in the previous section, we will now use charts as a stepping stone to dependent Bayesian lenses.
Indeed as in the case of ordinary lenses over a cartesian category, charts are just the fibrewise opposite of lenses, so we can define
\[\BayesLens(\Ca) = \int_{X \in \Ca} \Fam_\Ca(\Ca(I, X)\op)\op\]
and we are further justified in calling these lenses, because they are exactly a category of generalised lenses in the sense of \cite{spivak_lenses}.

\begin{remark}
In order to avoid confusion due to the name collision, we briefly distinguish between the existing definition of Bayesian lenses in \cite{statistical_games} and dependent Bayesian lenses as defined here, as these categories are related, but not the same.
Smithe's lenses have as objects, pairs of objects in $\Ca$, whereas dependent Bayesian lenses here have pairs $(X, A)$ where $X$ is also an object in $\Ca$, but $A$ is a $\Ca(I, X)$-indexed family of objects from $\Ca$.
Then, a morphism $\diset X S \lensto \diset Y R$ of ordinary Bayesian lenses consists of a morphism $f: X \to Y$ and a function $f^\sharp: \Ca(I, X) \to \Ca(R, S)$.
A morphism $\diset X A \lensto \diset Y B$ of dependent Bayesian lenses consists also of a morphism $f: X \to Y$ but then has as its backwards component a dependent function $f^\sharp: (\pi : \Ca(I, X)) \to \Ca(B(\pi\comp f), A(\pi))$.
So we see explicitly that dependent Bayesian lenses differ exactly by the addition of dependence to the function consisting the backwards component.

This relationship is further expounded when we consider the indexed categories used in defining either category of lenses.
The indexed category $\Stat$ in \cite{statistical_games} is very similar to $\Fam_\Ca(\Ca(I, -)\op)$.
Specifically, for each $X$, $\Stat(X)$ is the subcategory of $\Fam_\Ca(\Ca(I, X)\op)$ obtained by taking only constant families of objects.
\end{remark}

As before, where we saw that the category of charts has an oplax monoidal section $T$ restricting morphisms to families of support objects, we will show here that Bayesian lenses have a similar section sending morphisms to their inverses between supports.

\begin{proposition}
Assume that $\Ca$ has all support objects and Bayesian inverses.
The fibred category of dependent Bayesian lenses has a section $S: \Ca \to \BayesLens(\Ca)$.
\end{proposition}
\begin{proof}
On objects, $S$ has the same action as $T$, sending $X$ to the pair $\diset X {X_{(-)}}$ where $X_{(-)}: \Ca(I, X) \to \Ca$ picks out support objects of $X$ at each $\pi: I \to X$. 
On morphisms $f: X \to Y$ we have $S(f) = \diset f {f^\sharp}$ where $f^\sharp$ is the family of Bayesian inverses-with-support of $f$ at each prior.
It is easy to see that $S(\id_X) = \diset {\id_X} {\id_{X_{(-)}}}$,
so it remains to check that Bayesian inversion is functorial.
This follows for similar reasons to the functoriality of $T$.
If $f$ and $g$ are composable morphisms then $g^\sharp_{\pi\comp f} \comp f^\sharp_{\pi}$ is an inverse-with-support exactly when
$r_{\pi\comp f \comp g} \comp g^\sharp_{\pi\comp f} \comp f^\sharp_\pi \comp i_\pi$ is an ordinary Bayesian inverse.
But this is the composition of $r_{\pi\comp f \comp g} \comp g^\sharp_{\pi\comp f}\comp i_{\pi\comp f}$ and $r_{\pi\comp f}f^\sharp_\pi \comp i_\pi$ which must themselves be ordinary inverses.
And it is easy to check that composites of ordinary inverses must also be inverses.
\end{proof}

$\BayesLens(\Ca)$ has a monoidal product defined in the same way as to Bayesian charts.
The oplax monoidal structure given for $T$ also has a similar analogue for $S$, but the change in direction in the fibres means that it is instead lax monoidal, meaning that $S$ preserves monoid objects where $T$ preserved comonoids.

While the loss of functorial copying is unfortunate, it is easy to see why this should be the case.
A comultiplication morphism for such comonoids in Bayesian lenses should have backwards maps of the form
\[X_\pi \otimes X_\pi \xrightarrow{\delta_{X, X}} (X \otimes X)_{\pi\comp \Delta_X} \xrightarrow{\Delta_X^\sharp} X_\pi\]
but in most cases $X_\pi \otimes X_\pi$ should be thought of as being a lot bigger than $X_\pi$ and there is not a canonical way to multiply two arbitrary distributions on the same object.
On the other hand the Bayesian inverse $\Delta_X^\sharp: (X \otimes X)_{\pi\comp \Delta_X} \to X_\pi$ is canonical, because the distribution $\pi\comp\Delta_X$ is supported on a small subset of the points in $X \otimes X$.
Namely we should think of $(X \otimes X)_{\pi\comp \Delta_X}$ as being (a subspace of) the diagonal on $X$.
We make this precise in the following:

\begin{proposition}
Let $\pi: I \to X$ then $(X \otimes X)_{\pi\comp \dup_X} \cong X_\pi$, with an isomorphism given by the Bayesian inverse (with support) of copy.
\end{proposition}

\begin{proof}
In the following we write $C$ for $\dup_X: X \to X \otimes X$.
We will show that $C$ has an inverse, given by the Bayesian inversion of a (without loss of generality) left marginalisation morphism $L: X \otimes X \to X$.

Note that from the comonad laws we have $C\comp L = \id_X$, and consider the Bayesian inverse-with-support of $L$ at $\pi\comp C$.
This has a domain of $X_{\pi\comp C \comp L}$ which by the previous observation is equal to $X_\pi$.
So we have $L^\sharp_{\pi\comp C}: X_\pi \to (X\otimes X)_{\pi\comp C}$.
The definition of an inverse-with-support gives us that
\begin{center}
\includegraphics[scale=0.08]{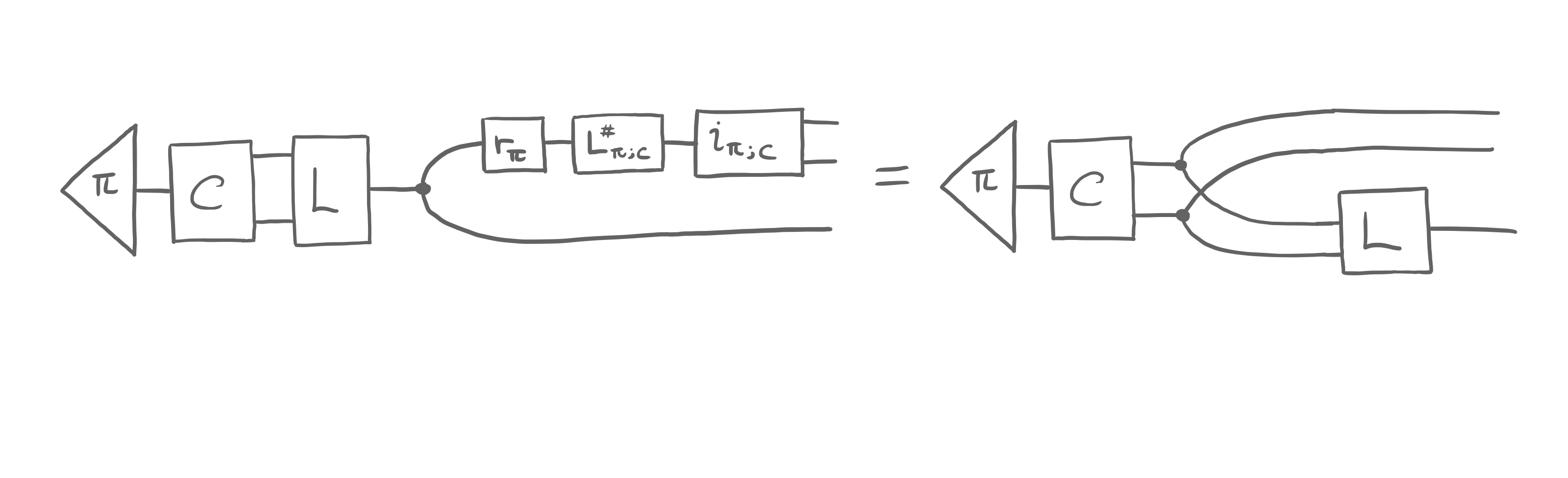}
\end{center}
Recalling that $C$ is just the copy morphism, $L$ is just a delete, and using the comonoid laws, this can be deformed to show that
$r_\pi\comp L^\sharp_{\pi\comp C}\comp i_{\pi\comp C} \asequal{\pi\comp C} C$,
\begin{center}
\includegraphics[scale=0.08]{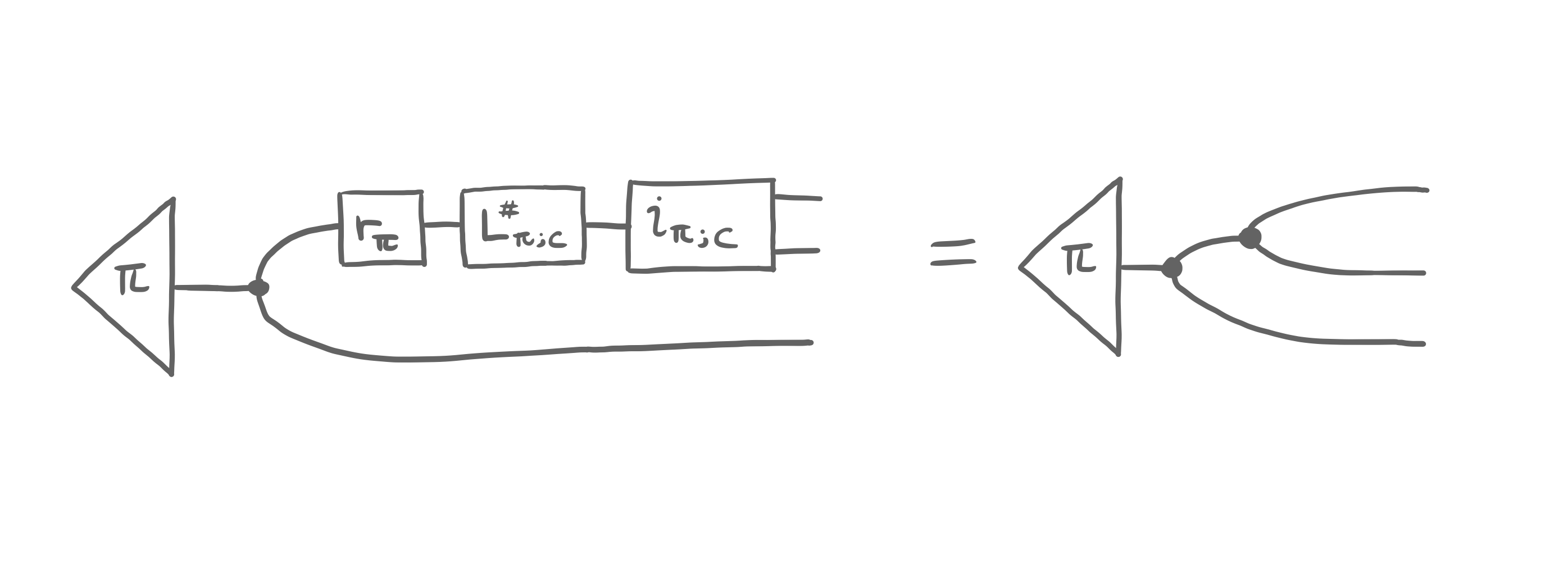}
\end{center}
Then, precomposing with $i_{\pi}$ we have
$L^\sharp_{\pi\comp C}\comp i_{\pi\comp C} = i_\pi\comp C$.
A similar argument using the inverse of $C$ at $\pi$ gives us that
$C^\sharp_\pi\comp i_\pi = i_{\pi\comp C} \comp L$.

Both equations can be drawn as commutative squares which paste together in two ways:
\[\begin{tikzcd}
	{(X\otimes X)_{\pi\comp C}} & {X_\pi} & {(X \otimes X)_{\pi\comp C}} && {X_\pi} & {(X \otimes X)_{\pi}} & {X_\pi} \\
	{X \otimes X} & X & {X \otimes X} && X & {X \otimes X} & X
	\arrow["{L^\sharp_{\pi\comp C}}", from=1-2, to=1-3]
	\arrow["{i_{\pi\comp C}}", from=1-3, to=2-3]
	\arrow["{i_\pi}"', from=1-2, to=2-2]
	\arrow["C"', from=2-2, to=2-3]
	\arrow["{i_{\pi\comp C}}"', from=1-1, to=2-1]
	\arrow["{C^\sharp_{\pi}}", from=1-1, to=1-2]
	\arrow["L"', from=2-1, to=2-2]
	\arrow["{L^\sharp_{\pi\comp C}}", from=1-5, to=1-6]
	\arrow["{C^\sharp_\pi}", from=1-6, to=1-7]
	\arrow["{i_\pi}"', from=1-5, to=2-5]
	\arrow["{i_{\pi\comp C}}"', from=1-6, to=2-6]
	\arrow["{i_\pi}", from=1-7, to=2-7]
	\arrow["C"', from=2-5, to=2-6]
	\arrow["L"', from=2-6, to=2-7]
\end{tikzcd}\]
Reading off the composite diagrams we have that
\begin{equation}
L^\sharp_{\pi\comp C} \comp C^\sharp_{\pi} \comp i_\pi = i_\pi \comp C \comp L
\end{equation}
\begin{equation}
C^\sharp_\pi \comp L^\sharp_{\pi\comp C} \comp i_{\pi\comp C} = i_{\pi\comp C} \comp L \comp C
\end{equation}
We already have that $C \comp L = \id_X$ so, postcomposing with $r_\pi$, equation (1) reduces to
$L^\sharp_{\pi\comp C}\comp C^\sharp_\pi = \id_{X_\pi}$.
Since $L$ and $C$ are only 1-sided inverses in general, equation (2) does not reduce as directly.
However, drawing the relevant string diagrams, it is easy to see that $L \comp C \asequal{\pi\comp C} \id_{X\otimes X}$:
\begin{center}
\includegraphics[scale=0.08]{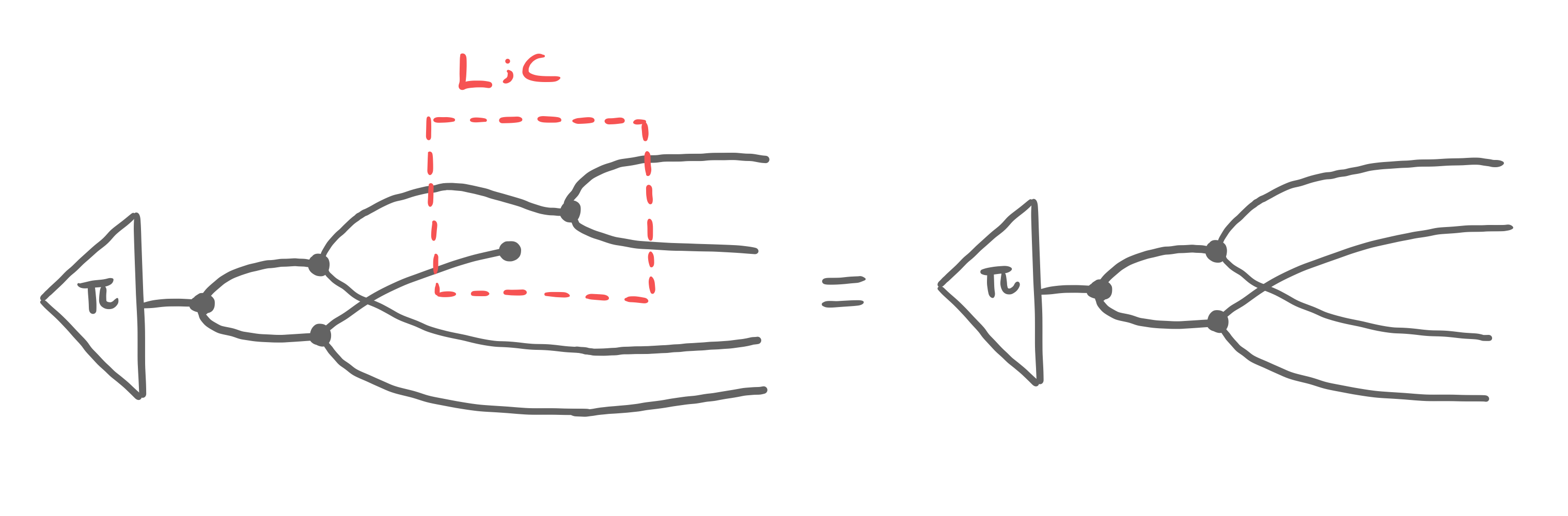}.
\end{center}
Hence we have
$i_{\pi\comp C}\comp L \comp C = i_{\pi\comp C}$
and so equation (2) reduces to
$C^\sharp_\pi\comp L^\sharp_{\pi\comp C} = \id_{(C \otimes C)_{\pi\comp C}}$.
\end{proof}

The fact that the Bayesian inverse to copying is defined only on the diagonal should make intuitive sense.
Indeed, if we had a stochastic process which performs a copy, we would expect that some observation of the output would agree on both copies.
Having support objects in the fibres allows us to encode a guarantee of this in the type system of the bidirectional processes modeled by dependent Bayesian lenses:
the Bayesian inverse of copying is an isomorphism witnessing the equality of its two inputs.
This was in fact one of the main motivations for defining them in this way.

\section{Further Work}

\subsection{Stochastic Dynamical Systems}

As mantioned earlier, the assembly of dependent Bayesian charts and lenses, along with their sections, fit nicely into Myers' categorical systems theory \cite{categorical_systems_theory}. \cite{djm_dynamical_systems} already explains how Markov decision problems can be formulated in this framework using a probability monad, but the way in which Bayesian lenses tie into this has not been explored.
We believe that various Bayesian filtering algorithms can be formulated in this way.

\subsection{Probabilistic Programming Languages}
A motivation for studying Bayesian lenses is that they provide a nice setting for discussing automatic Bayesian inversion.
Such a procedure could be implemented in a probabilistic programming language akin to automatic differentiation in other languages.
As such, we are led to wonder what an internal language of $\BayesLens(\Ca)$ could look like.

One option would be to work entirely within the image of the inversion functor $S$, essentially hiding the backwards pass from the programmer and generating it automatically, and possibly adding some additional `external' mechanisms to control sampling or updating of priors.

A less limiting approach would expose the backwards pass to the programmer allowing for customisation of numerical methods used or even making more of the lens category available so that the update morphisms can do more than just Bayesian inversion.
Indeed a similar language has been used to describe open games \cite{open_games} and has been implemented as a DSL embedded in Haskell \cite{open_game_engine}.
It is mentioned in the previous section how the inverse to copying can be seen to make use of dependent types to require a proof of equality in its input.
We expect there could be further applications of the distribution-dependent types modelled by dependent Bayesian lenses and charts if they were exposed in the type system of such a language.

We expect there may be a middle ground between the simplicity of a language with automatic inversion and the expressiveness of the Open-Game Engine which allows the user to make use of the richness of the category of Bayesian lenses while still allowing for ergonomic Bayesian updating.
For example \cite{exact_conditioning} describes a functional language with an operation allowing the user to condition distributions on terms in the language with semantics in a copy-delete category $\mathbf{Cond}(\Ca)$ constructed from a Markov category $\Ca$.
$\mathbf{Cond}(\Ca)$ differs from lenses in that conditioning makes use of effects, such that Bayesian updating changes the semantics of terms, but the string diagrams representing inversion seem similar to diagrams for lenses.
We intend to investigate whether this (or a similar language) could be given semantics using Bayesian lenses, and whether there is a relationship between the semantic categories.

\begin{figure*}[h!]
\centering
\includegraphics[scale=0.15]{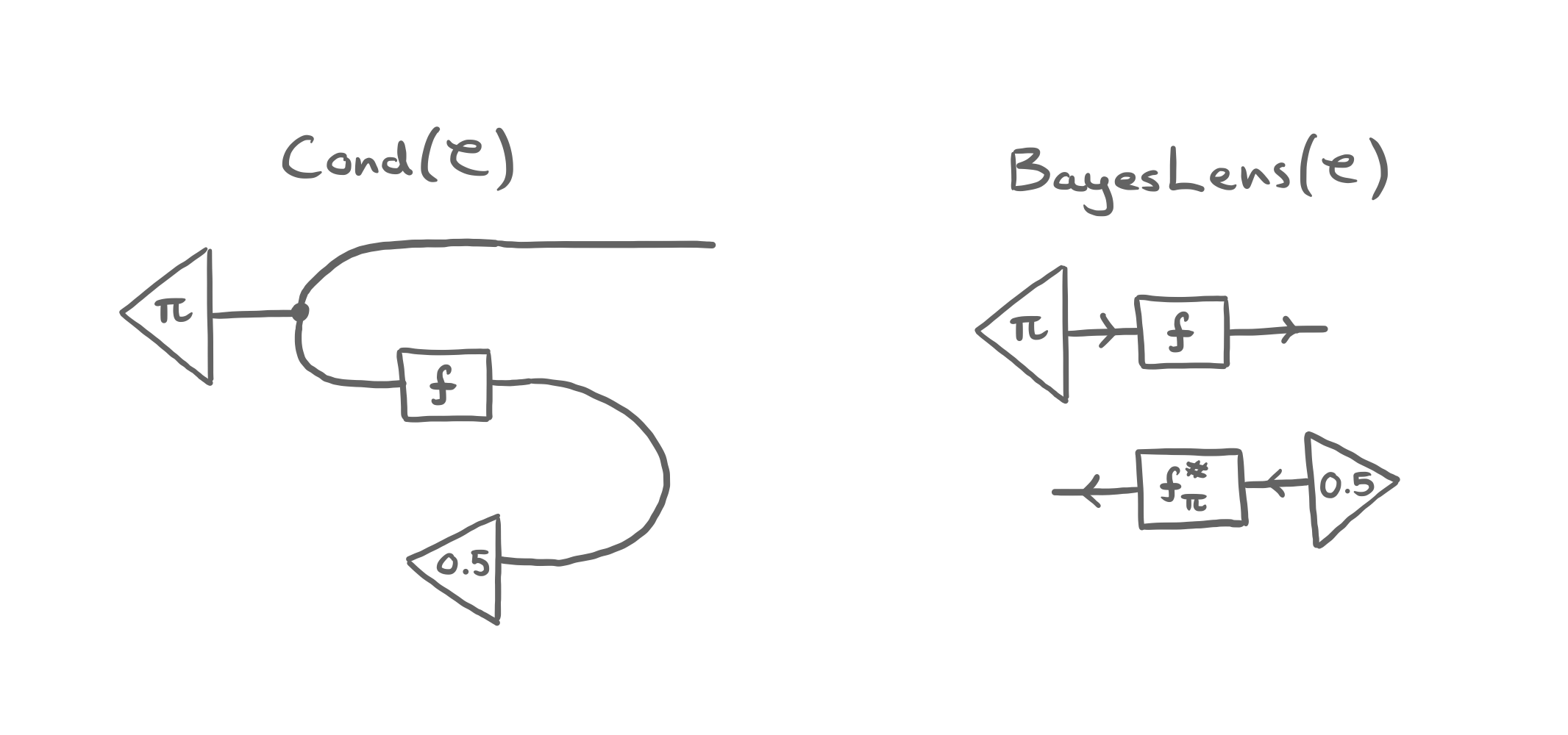}

Figure: Diagrams representing the updated distribution $f^\sharp_\pi(0.5)$ in the categories $\mathbf{Cond(\Ca)}$ and $\BayesLens(\Ca)$ respectively
\end{figure*}

\subsection{Dependent Optics}

Recent works
\cite{fibre_optics}\cite{dependent_optics}\cite{seeing_double}
have proposed definitions for categories of \emph{dependent optics} which simultaneously generalise optics and functor-lenses.
However the inclusion of functor lenses into dependent optics is in some ways ``operationally'' unsatisfactory.
That is to say, the canonical way to obtain dependent lenses from an $F$-lens, gives lenses where the fibres of the category in the forward direction are trivial.

This triviality in the optic representation of $F$-lenses limits our ability to view the operation of Bayesian lenses `optically' where we would have transformations in two directions sharing data via a residual object.
Such concerns are relevant to the complexity of implementations of Bayesian lenses.
However, there may be alternative ways to represent specific examples of $F$-lenses (i.e. where $F$ is a fixed functor) as dependent optics.
Indeed regular lenses over a cartesian category can be represented as dependent optics in two ways: one as a representation of optics with the cartesian monoidal product, and the other the a representation of lenses realised as $F$-lenses where $F$ is an indexed category over a category of comonoids which gives rise to ordinary lenses.

This leads us to the question of whether Bayesian lenses can be represented in another way which is more operationally insightful.

\subsection{Categories with Supports}

We make heavy use of categories where we have assumed the existence of support objects and Bayesian inversion.
However we do not have many good examples of Markov categories satisfying these properties.
We mentioned already that $\mathbf{FinStoch}$ and $\mathbf{Gauss}$ are such examples but we would like to find further examples that may be used in this framework:
possibly categories of certain `nice' measure spaces or other families of distributions like $\mathbf{Gauss}$.

\bibliography{generic}

\end{document}